\documentclass{amsart}
\usepackage[utf8]{inputenc}
\usepackage{amsfonts,amssymb,amsthm}
\usepackage{mathtools,verbatim}
\usepackage{enumerate,xcolor,hyperref}
\usepackage[foot]{amsaddr}
\makeatletter\@namedef{subjclassname@2020}{\textup{2020} Mathematics Subject Classification}\makeatother
\newtheorem{theorem}{Theorem}[section]
\newtheorem{lemma}[theorem]{Lemma}
\newtheorem{corollary}[theorem]{Corollary}
\newtheorem{proposition}[theorem]{Proposition}
\theoremstyle{definition}

\newtheorem{remark}[theorem]{Remark}
\newtheorem{example}[theorem]{Example}
\numberwithin{equation}{section}
\numberwithin{table}{section}
\numberwithin{figure}{section}
\newcommand{\tr}{\operatorname{tr}}

\newcommand{\Comp}{\mathbb{C}}

\DeclareMathOperator{\ii}{i}
\DeclareMathOperator{\var}{var}
\DeclareMathOperator{\mean}{mean}

\newcommand{\BC}{{\mathbb C}}
\newcommand{\CC}[1]{{\mathbb C}^{#1\times #1}}
\newcommand{\BN}{{\mathbb N}}

\newcommand {\mat}  [1] {\left[\begin{array}{#1}}
\newcommand {\rix}      {\end{array}\right]}
\newcommand{\norm}[1]{ \left\| #1 \right\| }
\newcommand{\bignorm}[1]{ \biggl\| #1 \biggr\| }

\makeatletter
\font\tenex=cmex10 
\newdimen\p@renwd
\setbox0=\hbox{\tenex B} \p@renwd=\wd0 
\def\bmat#1{\begingroup \m@th
  \setbox\z@\vbox{\def\cr{\crcr\noalign{\kern2\p@\global\let\cr\endline}}%
    \ialign{$##$\hfil\kern2\p@\kern\p@renwd&\thinspace\hfil$##$\hfil
      &&\quad\hfil$##$\hfil\crcr
      \omit\strut\hfil\crcr\noalign{\kern-\baselineskip}%
      #1\crcr\omit\strut\cr}}%
  \setbox\tw@\vbox{\unvcopy\z@\global\setbox\@ne\lastbox}%
  \setbox\tw@\hbox{\unhbox\@ne\unskip\global\setbox\@ne\lastbox}%
  \setbox\tw@\hbox{$\kern\wd\@ne\kern-\p@renwd\left[\kern-\wd\@ne
    \global\setbox\@ne\vbox{\box\@ne\kern2\p@}%
    \vcenter{\kern-\ht\@ne\unvbox\z@\kern-\baselineskip}\,\right]$}%
  \null\;\vbox{\kern\ht\@ne\box\tw@}\endgroup}
\makeatother

\def\real{\operatorname{Re}}
\def\imag{\operatorname{Im}}

\title[The Sz\'asz inequality for matrix polynomials]{The Sz\'asz inequality for matrix polynomials \mbox{and functional calculus}}

\author[P. Pikul]{Piotr Pikul}
\address{Piotr Pikul\\Instytut Matematyki,
Wydzia\l\ Matematyki i~Informatyki\\Uniwersytet Jagiello\'{n}ski\\
ul.\ \L{}ojasiewicza 6\\30-348 Krak\'{o}w\\Poland\\
ORCiD: 0000-0001-7461-0248}
\email{piotr.pikul@im.uj.edu.pl}
\author[O.\,J. Szymański]{Oskar Jakub Szymański}
\address{Oskar Jakub Szymański\\Instytut Matematyki,
Wydzia\l\ Matematyki i~Informatyki\\Uniwersytet Jagiello\'{n}ski\\
ul.\ \L{}ojasiewicza 6\\30-348 Krak\'{o}w\\Poland}
\email{oskar.szymanski@doctoral.uj.edu.pl}
\author[M. Wojtylak]{Michał Wojtylak}
\address{Michał Wojtylak\\Instytut Matematyki,
Wydzia\l\ Matematyki i~Informatyki\\Uniwersytet Jagiello\'{n}ski\\
ul.\ \L{}ojasiewicza 6\\30-348 Krak\'{o}w\\Poland\\
ORCiD: 0000-0001-8652-390X}
\email{michal.wojtylak@uj.edu.pl}

\begin{document}
\begin{abstract}
The Szász inequality is a classical result that provides a bound for polynomials with zeros
in the upper half of the complex plane, expressed in terms of their low-order coefficients.
Generalizations of this result to polynomials in several variables have been obtained by Borcea-Brändén and Knese.

In this article, we discuss the Szász inequality in the context of polynomials
with matrix coefficients or matrix variables. In the latter case, the estimation provided by
the Szász-type inequality can be sharper than that offered by the von Neumann inequality.

As a byproduct, we improve the scalar Szász inequality by relaxing the assumption regarding the location of zeros.

\end{abstract}

\maketitle

\subsubsection*{Statements and Declarations} No potential conflict of interest was reported by the authors.
\section{Introduction}

A polynomial is called {\em stable} if all its zeros lie outside the open upper half-plane.
O.~Sz\'asz \cite{szasz1943} discovered an inequality bounding a stable
polynomial $p(\lambda) = a_d \lambda^d + a_{d-1} \lambda^{d-1} + \ldots + a_1 \lambda  + 1 \in \Comp[\lambda]$
in terms of its first few coefficients:
\[
|p(\lambda )| \leq \exp\bigl(|a_1| |\lambda | + 3(|a_1|^2 + |a_2|)|\lambda |^2\bigr).
\]

Later on, the inequality was improved by de Branges \cite[Lemma 5]{deB61} to
\begin{equation}\label{dB}
|p(\lambda )| \leq \exp\bigl(\real(a_1 \lambda)  + {\textstyle\frac12} (|a_1|^2 -2 \real(a_2))|\lambda |^2\bigr),\quad \lambda \in\Comp.
\end{equation}
Knese in \cite[Theorem 1.3]{Kne19} provided a simplified proof and showed the sharpness
of the bound on the imaginary axis. The elegant presentation in \cite{Kne19} actually
shows that stability is not a necessary condition for \eqref{dB} to hold.
We will elaborate on that in Theorem~\ref{scalar+semis}, Proposition~\ref{pro-random},
and Examples~\ref{ex:semis1} and \ref{ex:semis2}.

Except for its own beauty, the Sz\'asz inequality is worth considering as it is applied in several places in analysis.  
It assures normality of some families of entire functions. As a consequence, 
it allows for a complete characterisation of the local uniform limits of stable polynomials of one variable \cite[Chapter VIII, Theorem 4]{Lev80}. Furthermore, Borcea–Brändén \cite{BorB09} used a multivariate version of the inequality to characterise linear operators on the space of polynomials that preserve stability. 
Having said this, it seems natural to ask how this inequality behaves for matrix polynomials and matrix variables.

The first result for matrix polynomials was already established in \cite{szymanski2023stability} as Proposition 4.5. Namely, if the numerical range of the matrix polynomial $P(\lambda)$ is contained in some half-plane $H_\varphi$, $\varphi \in [0,2\pi)$, then 
\begin{equation}\label{lH}
\| P(\lambda) \| \leq 2\exp \Bigl[
\lambda_H\bigl(\lambda A_1 -|\lambda|^2 A_2\bigr)
+\frac12 |\lambda|^2 \norm{A_1}^2 
\Bigr], \quad \lambda \in \Comp,
\end{equation}
where $\lambda_H(X)$ denotes the largest eigenvalue of the matrix $\real X=\frac{X+X^*}2$ (equivalently:
the maximal real coordinate of the numerical range).

As shown in Example~\ref{ex:hyperstable} below, in general, (hyper)stability of a matrix polynomial
is not sufficient to ensure a bound of the norm based exclusively on the low-order coefficients.
However, it appears that by additionally assuming a factorisation of the polynomial into degree-one terms, we obtain the inequality
\[
\norm{P(\lambda)}_F \leq 
n^{\frac d2}\exp\Bigl(\frac{1}{n}\tr\real(\lambda A_1)
+ \frac{1}{2n}\big(\norm{A_1}_F^2 - 2\tr\real A_2\big)|\lambda|^2\Bigr),
\]
see Theorem~\ref{thm-fact1} below.

In this way, our work connects with the topic considered extensively by Harm Bart, Rien Kaashoek, and Israel Gohberg; see \cite{BGK1979,bart1988complementary} and \cite{bart2007factorization}. The latter book contains a seminal formulation of necessary and sufficient conditions for a rational matrix-valued function $W(\lambda)$ to have a factorisation into rank-one elementary factors. We apply this approach to $W(\lambda)=P(1/\lambda)$, obtaining a closed form of the Sz\'asz inequality in terms of its
system realisation 
\[
P(\lambda)=I+\lambda C(I-\lambda A)^{-1} B,
\]
see Theorem~\ref{ABCISzasz}.

Our work on functional calculus is motivated by the von Neumann inequality
\[
\norm{ p(T)}\leq \sup_{|z|\leq \norm T} |p(z)|,
\]
which holds for any polynomial $p$ (matrix or scalar) and any operator $T$.
Despite a large interest in generalising this result
(see \cite{Kosin,Kne16,hartz2023ax,HarR22} for the works most related to the current paper), there is little known about a possible interplay between the properties of $p$ and $T$ that might lead to an improved bound.
In Theorems~\ref{pro-funct1} and \ref{pro-funct2} below,
we estimate the norm of $p(A)$ for a matrix $A$ and a stable
scalar polynomial $p$ in three independent ways:
\begin{align*}
\norm{p(A)}_F \leq&\
n^{\frac d2}\exp\Bigl(\frac{1}{n}\tr\real(a_1 A)
+ \frac{1}{2n}(|a_1|^2 - 2\real a_2)\norm A _F^2\Bigr), \\
\norm{p(A)} \leq&\ \exp\Bigl(\norm A \sqrt{d(|a_1|^2 - 2\real a_2)}\Bigr), \\
\norm{p(A)} \leq&\ \exp\bigl(|a_1|\norm A  + {\textstyle\frac12} (|a_1|^2 -2 \real(a_2))\norm A^2\bigr).
\end{align*}
The first of these inequalities presents the aforementioned interplay between $T$ and the coefficients of $p$. Specifically, since $-a_1$ is the sum of the reciprocals of the roots of $p(\lambda)$, the factor $\exp(\operatorname{tr} \operatorname{Re}(a_1 A))$ might significantly improve the bounds due to the location of the roots and the eigenvalues of $A$.
We demonstrate in Example~\ref{comparison} that the first inequality provides in some cases a better estimation than the von Neumann inequality.


In addition to these considerations, in the last section of the paper we briefly extend the results to scalar and matrix polynomials in commuting matrix variables.

 \section{Preliminaries}

Throughout this article we denote the field of complex numbers by $\Comp$.
By $\BN$ we mean positive 
integers.
Symbol $\CC{n}$ stands for the set of square $n\times n$ matrices with
complex entries. The set of matrix polynomials of $k$ variables with
square $n\times n$ coefficients will be denoted by $\CC{n}[z_1,\ldots,z_k]$.
Matrix polynomials will be denoted by capital letters to be easily distinguishable from
scalar ones.

A scalar polynomial $p\in\Comp[z_1,\ldots,z_k]$ is called stable with respect to
the set $D\subseteq \Comp^k$ if all its zeros lie outside $D$. If the set
is not specified explicitly, we assume it is a product of upper half planes
(i.e.\ $D=H_0^k=\{z\in\Comp^k: \imag(z_j)\geq 0,\ j=1,\ldots,k\}$).
For a matrix polynomial $P\in\CC{n}[\lambda]$ stability with respect to $D\subseteq \Comp$ means that
$P(\lambda)$ is invertible for every $\lambda \in D$. Point $\lambda\in\Comp$ that
$\ker P(\lambda)\neq\{0\}$ is called \emph{eigenvalue} of the matrix polynomial $P$.
The numerical range of $P(\lambda)$ is defined as
$$
\{\lambda\in\Comp: x^* P(\lambda)x=0\text{ for some } x\in\Comp^n\setminus\{0\}\}.
$$

For a matrix $A\in \CC{n}$ we denote its standard operator norm
by $\|A\|:=\sup\big\{\|Ax\|: x\in\Comp^n,\ \|x\|=1\big\}$.
We will also use the Frobenius norm $\norm{A}_F:=\sqrt{\tr (A^*A)}$
and the matrix imaginary part (skew part of the matrix) $\imag A:=\frac{1}{2i}(A-A^*)$. 
For a vector $z\in \Comp^n$ we denote its standard euclidean norm by $\norm{z}$, and $\norm{z}_\infty$
stands for the supremum norm.

We provide now several lemmas, which will be used later on. Note that for
$A \in \CC{n} \setminus \{-I\}$ one easily obtains 
\begin{equation}\label{imm}
 \log\norm{I + A}\leq \norm A. 
\end{equation}
We present now a more subtle estimate.
\begin{lemma}\label{mlog}
 Let $A \in \CC{n} \setminus \{-n^{-1/2}I\}$. Then
\[
\log\norm{\frac{1}{\sqrt{n}}I + A}_F \leq \frac{1}{\sqrt{n}}\tr\real A + \frac{1}{2}\norm A _F^2\text{.}
\]    
\end{lemma}
\begin{proof} The proof is based on a direct computation:
\begin{align*}
\log\norm{\frac{1}{\sqrt{n}}I + A}_F
&= \frac{1}{2}\log\norm{\frac{1}{\sqrt{n}}I + A}_F^2 \\
&=\frac{1}{2}\log\tr\Bigl[\Bigl(\frac{1}{\sqrt{n}}I + A\Bigr)^*\Bigl(\frac{1}{\sqrt{n}}I + A\Bigr)\Bigr] \\
& = \frac{1}{2}\log\tr\Big(\frac{1}{n}I + \frac{1}{\sqrt{n}}A + \frac{1}{\sqrt{n}}A^* + A^*A\Big)\\
&= \frac{1}{2}\log\Big(1 + \frac{2}{\sqrt{n}}\tr\real A + \norm A _F^2\Big)
\end{align*}
Now it is enough to use the scalar inequality $\log(1 + x) \leq x$, $x \in (-1; \infty)$.
\[
\log\norm{\frac{1}{\sqrt{n}}I + A}_F
\leq  \frac{1}{2}\Big(\frac{2}{\sqrt{n}}\tr\real A + \norm A _F^2\Big)
= \frac{1}{\sqrt{n}}\tr\real A + \frac{1}{2}\norm A _F^2\text{.}
\]
\end{proof}

Knese showed in \cite[Proof of Lemma 2.1]{Kne19} that for $b_1,\dots,b_d\in\Comp$ satisfying
$\sum_{1 \leq j < k \leq d}\imag b_j\imag b_k\geq 0$ (equivalently, satisfying \eqref{var-semis} below) one has
\begin{equation}\label{scalarv}
    \sum_{j=1}^d |b_j|^2 \leq
\Big|\sum_{j=1}^d b_j\Big|^2 - 2\real\sum_{1 \leq j < k \leq d}b_jb_k.
\end{equation}
We present a matrix version of this inequality.

\begin{lemma}\label{sums}
Let $B_1, B_2, \dots , B_d$ be $n \times n$ complex matrices such that
\begin{equation}\label{sumim}
\sum_{1 \leq j < k \leq d}\tr(\imag B_j\imag B_k)\geq 0.
\end{equation}
Then
\[
\sum_{j=1}^d \norm{B_j}_F^2 \leq
\bignorm{\sum_{j=1}^d B_j}_F^2 - 2\tr\real\sum_{1 \leq j < k \leq d}B_jB_k\text{.}
 \]
\end{lemma}
\begin{proof}
Observe that
\begin{align*}
\bignorm{\sum_{j=1}^d B_j}_F^2
&= \tr\biggl(\sum_{j=1}^d B_j^*\sum_{k=1}^d B_k\biggr)\\
&= \tr\biggl(\sum_{j=1}^d B_j^*B_j + \sum_{1 \leq j \neq k \leq d}B_j^*B_k\biggr)  \\
&= \sum_{j=1}^d \norm{B_j}_F^2 + 2\tr\real\sum_{1 \leq j < k \leq d}B_j^*B_k.
\end{align*}
Therefore, it is sufficient to prove the following statement
\[
2\tr\real\sum_{1 \leq j < k \leq d}(B_j^*B_k - B_jB_k) \geq 0 .
\]
To do so, note that
\begin{align*}
2\tr \real & \sum_{1 \leq j < k \leq d}(B_j^*B_k - B_jB_k) \\
&= 
\sum_{1 \leq j < k \leq d}\bigl(\tr(B_j^*B_k) - \tr(B_jB_k) + \tr(B_k^*B_j) - \tr(B_k^*B_j^*)\bigr) \\
&= \sum_{1 \leq j < k \leq d}\tr\bigl((B_j - B_j^*)(B_k^* - B_k)\bigr) \\
&= 4\sum_{1 \leq j < k \leq d}\tr(\imag B_j\imag B_k),
\end{align*}
which is a non-negative number by assumption. 
\end{proof}

For a better understanding of the assumption \eqref{sumim} we present two lemmas.

\begin{lemma}\label{minus-semis}
Let $B_1, B_2, \dots , B_d$ be $n \times n$ complex matrices such that
$\imag B_j \leq 0$ for $j \in \{1, 2, \dots , d\}$.
Then 
$$
\sum_{1 \leq j < k \leq d}\tr(\imag B_j\imag B_k)\geq 0.
$$
\end{lemma}

\begin{proof}
Observe that 
\[
\tr(\imag B_j\imag B_k) =
\tr\bigl(
\sqrt{-\imag B_j} (-\imag B_k) \sqrt{-\imag B_j}
\bigr) \geq 0
\]
as the matrix $\sqrt{-\imag B_j}(-\imag B_k)\sqrt{-\imag B_j}$ is positive semi-definite
due to the assumption that $\imag B_j$ and $\imag B_k$ are both nonpositive semi-definite matrices.
    
\end{proof}

The following elementary proposition elaborates on the expression
in the assertion of the previous lemma.
\begin{proposition}\label{elem}
Let $B_1, B_2, \dots , B_d$ be $n \times n$ complex matrices, then
\begin{equation}\label{trImB}
    2\sum_{1 \leq j < k \leq d}\tr(\imag B_j\imag B_k)=\Big\| \sum_{j=1}^d \imag B_j\Big\| ^2_F-\sum_{j=1}^d \norm{\imag B_j}^2_F.
\end{equation}
In particular\begin{enumerate}[\rm (a)]
\item for rank-one matrices $B_j=u_jv_j^*$, $j=1,\dots, d$, one has
\begin{equation}\nonumber\label{trIm-uv}
2
\tr\big(\imag(u_j v_j^*) \imag (u_k v_k^*)\big)
= 
\real\Big(
 (u_j^*u_k) (v_k^*v_j )
- (v_j^*u_k)  (v_k^*u_j)
\Big);
\end{equation}
\end{enumerate}

\end{proposition}
\begin{proof}
The proof of identity \eqref{trImB} is a straightforward calculation.
For part (a) one uses the equalities of the form $\tr(v_kv_j^*u_ju_k^*+v_jv_k^*u_ku_j^*)=2\real(u_k^*v_kv_j^*u_j)$. Part (b) is immediate.

\end{proof}
\section{Scalar case}

First we present a weaker version of the stability  assumption.
Before we will use this weaker form extensively in our matrix versions, we discuss the scalar case in detail.

Let $p(\lambda)=\prod_{j=1}^d (1+b_j\lambda)$ be a scalar polynomial. Consider the condition
\begin{equation}\label{semis}
\sum_{1\leq j<k\leq d} \imag b_j \imag b_k
\geq 0.
\end{equation}
By elementary calculation, this is  equivalent to 
\begin{equation}\label{var-semis}
\operatorname{var} \imag \Lambda  \leq (d-1) (\operatorname{mean}\imag \Lambda) ^2, \quad \Lambda=\left(b_1,\dots b_d\right).
\end{equation}
Clearly, any polynomial stable with respect to the lower or upper half plane satisfies \eqref{semis}.

 \begin{theorem}\label{scalar+semis}
For a polynomial  $p(\lambda)=1+\sum_{j=1}^d a_j\lambda^d$ satisfying  \eqref{var-semis} the
inequality \eqref{dB} holds, i.e.
\[
|p(\lambda)| \leq
\exp\bigl(\real(a_1 \lambda)+ (|a_1|^2 - 2\real a_2)|\lambda|^2\bigr),\quad \lambda \in \Comp.
\]
\end{theorem}

\begin{proof}
The proof follows directly from the proof of \cite[Theorem 1.3]{Kne19},
it is also a consequence of Theorem~\ref{thm-fact1} below.
\end{proof}

For $p$ of degree not greater than $2$, \eqref{var-semis} implies stability with respect to the
upper or lower half-plane. 
For $d\geq 3$ this is a generalisation of Sz\'asz inequality onto a broader class of polynomials.
Below we provide an example of a polynomial satisfying \eqref{var-semis} which is stable with respect to 
no half-plane.
\begin{example}\label{ex:semis1}
Consider polynomials of the form 
\[ p_c(\lambda):= \Big(1+\frac i{c+1}\lambda\Big)\Big(1-\frac{i+1}{c}\lambda\Big)^c\Big(1-\frac{i-1}{c}\lambda\Big)^c,\]
where $c\in\BN$. The roots are clearly $\{(c+1)i, \frac{c(1-i)}2, \frac{c(-1-i)}2\}$ and are contained
in no half-plane $\{\lambda\in\BC\colon\real e^{\ii\theta} \lambda>0\}$, $\theta\in[0,2\pi)$. Condition \eqref{var-semis} can be readily verified, since among the $2c+1$ imaginary
parts involved there are $2c$ values $-\frac1c$ and a single value $\frac1{c+1}$. The sum in \eqref{semis}
equals $\frac{2c-1}{c}-\frac2{c+1}=\frac{2c^2-c-1}{c^2+c} \geq 0$.
Expanding $p_c(\lambda)=1+a_1\lambda+a_2\lambda+\ldots+a_{2c+1}\lambda^{2c+1}$ we have
$a_1=(\frac1{c+1}-2)i$ and $a_2=\frac{-2c}{c+1}$.
    
\end{example}

Note that \eqref{semis} can be satisfied even if the numbers of zeros on each side
of the real axis are equal.

\begin{example}\label{ex:semis2}
Choosing $k$ roots satisfying $\imag\frac1{\lambda_1}=\ldots=\imag\frac1{\lambda_k}=1$
and the other $k$, with $\imag\frac1{\lambda_{k+1}}=\ldots=\imag\frac1{\lambda_{2k}}=-\frac{k+\sqrt{2k-1}}{k-1}$
results in fulfilling \eqref{semis}. Indeed, the  left-hand-side is equal to
\begin{align*}
\binom k2 1^2 + &\binom k2 \bigg(\frac{k+\sqrt{2k-1}}{k-1}\bigg)^2 - k^2 \frac{k+\sqrt{2k-1}}{k-1}\\
&= \frac{k}{2}\bigg((k-1)+\frac{(k+\sqrt{2k-1})^2}{k-1}  \bigg) - \frac{k^3+k^2\sqrt{2k-1}}{k-1}\\
&=  \frac{k}{2}\frac{2k^2+2k\sqrt{2k-1}}{k-1} - \frac{k^3+k^2\sqrt{2k-1}}{k-1} = 0.
\end{align*}
\end{example}

We conclude the comment on the condition \eqref{var-semis} with the following property.
\begin{proposition}\label{pro-random}
Assume that we draw an infinite sequence $\{b_j\}_{j=1}^\infty$ of
complex numbers independently from a probability distribution on the complex plane
with a non-real mean and a finite variance.
Then with probability $1$ there exists $d\in\BN$ such that the polynomial
$p(\lambda)=\prod_{j=1}^d(1+\lambda b_j)$ satisfies 
the global Sz\'asz inequality \eqref{dB}.
\end{proposition}
\begin{proof} We may assume without loss of generality that $b_j\neq 0$ for $j=1,\dots,d$
as the zero $b_j$'s do not contribute to the formula for $p(\lambda)$ and the mean of the underlying distribution is nonzero.  
With probability $1$ the value $\operatorname{var}(\imag b_1,\ldots,\imag b_d)$ tends to the variance
$\sigma^2>0$ of the imaginary part of the random variable.
Analogously, $\frac1d\imag\sum_{j=1}^d b_j$ tends to a nonzero value $\beta$ (imaginary part of the mean of
the random variable).
There exists $d_0\in\BN$ such that for $d\geq d_0$ we have simultaneously
\begin{align*}
\operatorname{var}(\imag b_1,\ldots,\imag b_d)\leq& \sigma^2+|\beta|/2,\\
\Big|\frac1d\imag\sum_{j=1}^d b_j\Big|\geq &|\beta|- |\beta|/2=|\beta|/2,\\
\sigma^2+|\beta|/2 \leq& ( d-1)(|\beta|/2)^2.
\end{align*}
Combining the above inequalities,  
we obtain
\begin{align*}
\operatorname{var}(\imag b_1,\ldots,\imag b_d) &
 \leq \sigma^2+|\beta|/2\\
&\leq ( d-1)(|\beta|/2)^2\\
&\leq ( d-1)\Big(\frac1{ d}\imag\sum_{j=1}^{ d} b_j\Big)^2.
\end{align*}
This shows that  \eqref{var-semis} holds, application of Corollary~\ref{scalar+semis} finishes the proof. 
\end{proof}

\section{Matrix polynomials in one variable}
In \cite{szymanski2023stability} the authors asked whether in \cite[Proposition 4.5.]{szymanski2023stability}
the assumption regarding numerical range can be replaced by hyperstability of $P$.
Let us recall that a matrix polynomial $P\in\CC{n}[\lambda]$ is called \emph{hyperstable} with respect to $D$
if for every $x\in \Comp^n\setminus\{0\}$ there exists $y\in \Comp^n\setminus\{0\}$ such that
$y^*P(\lambda)x\neq 0$ for any $\lambda\in D$. A~matrix polynomial with numerical range contained in a half-plane
is always hyperstable, but not conversely, cf.\ \cite{szymanski2023stability}.
We give a negative answer to the question above.

\begin{example}\label{ex:hyperstable}
Consider the matrix polynomial 
$$
P(\lambda)=\begin{bmatrix}
    1 & \lambda^d \\ 0 & 1
\end{bmatrix},
$$
with $d \geq 3$. Then $P(0)=I_2$, and $P(\lambda)$ is hyperstable with respect to any set $D\subseteq \Comp$ (cf. \cite[Proposition~3.3]{szymanski2023stability}). However,
$$
\norm{P(\lambda)} \geq \frac{1}{\sqrt 2} \norm{P(\lambda)}_F= \sqrt{1+\frac{|\lambda|^{2d}}{2}}.
$$
Hence, as $A_1 = A_2 = 0$, there can not be any global bound on  $\norm{P(\lambda)}$ of Sz\'asz type.

Let us also compute the numerical range of $P(\lambda)$.
Let $x \in \Comp^2\setminus\{0\}$. We have $x^*P(\lambda)x = |x_1|^2 + |x_2|^2 + \lambda^d\overline{x}_1 x_2$.
In particular, for $x_1 = x_2 = 1$, we see that the numerical range of $P(\lambda)$ is not contained in any half-plane. 
 \end{example}

The example above shows that hyperstability is not sufficient to obtain any bound depending only on $A_1$ and $A_2$.
On the other hand, requiring that the numerical range of the polynomial
lies in a half-plane, as in \eqref{lH}, is a strong assumption.
Therefore, we now provide a statement that requires a factorization $P(\lambda) = \prod_{j=1}^d(I + \lambda B_j)$,
instead of the numerical range condition. Note that the number of factors $d$ is not less than the degree of $P(\lambda)$.
The example $I_n + \lambda I_n = \prod_{j=1}^n (I_n + \lambda e_j e_j^*)$ shows that the number of factors might be
larger than the degree ($e_j$'s are vectors of the canonical basis in $\Comp^n$).

In Section~\ref{sec:factor} we will connect the  factorisation into  degree one factors with a system realisation result. From other related results on factorisation into degree one factors, we mention the relation to the numerical range by Li--Rodman \cite[Section 3]{LiR94}  and  a sufficient condition 
by Krupnik  \cite{krupnik}.


\begin{theorem}\label{thm-fact1} 
Let $P(\lambda) = I + \sum_{j=1}^N \lambda^j A_j
\in \CC{n}[\lambda]$ be a matrix polynomial with the following factorisation
$P(\lambda) = \prod_{j=1}^d(I + \lambda B_j)$, where  
\begin{equation}\label{M-semis}
\sum_{1\leq j<k\leq d}\tr(\imag B_j\imag B_k) \geq 0.\end{equation}
Then the Frobenius norm of
$P(\lambda)$ can be estimated from above as
\begin{equation}\label{Sz2}
\norm{P(\lambda)}_F \leq
n^{\frac d2}\exp\Bigl(\frac{1}{n}\tr\real(\lambda A_1)
+ \frac{1}{2n}\big(\norm{A_1}_F^2 - 2\tr\real A_2\big)|\lambda|^2\Bigr), 
\quad \lambda \in \Comp.
\end{equation}
\end{theorem}
 According to Lemma~\ref{minus-semis},
if $\imag B_j\leq 0$ for all $j=1,\ldots,d$ then the assumption \eqref{M-semis} is satisfied.
Moreover, the eigenvalues of the polynomial $P$ are then precisely of the form $-\frac1{\lambda}$
for $\lambda\in\sigma(B_j)$.  Hence,   $P$ is stable (with respect to the open half-plane).

\begin{proof}
First, we use Lemma~\ref{mlog}:
\begin{align}\label{nowenowe}
\log\norm{P(\lambda)}_F &\leq \sum_{j=1}^d\log\norm{I + \lambda B_j}_F \\
\nonumber &= d\log\sqrt{n} + \sum_{j=1}^d\log\norm{\frac{1}{\sqrt{n}}I + \lambda \tilde{B}_j}_F\\
\nonumber & \leq \log n^{\frac d2} + \frac{1}{\sqrt{n}}\sum_{j=1}^d \tr\real(\lambda \tilde{B}_j)
    + \frac{1}{2}\sum_{j=1}^d \norm{\lambda \tilde{B}_j}_F^2\text{,}
\end{align}
where $\tilde{B}_j := n^{-1/2}B_j$ for $j \in \{1,2, \dots , d\}$.
Since $I + \sum_{j=1}^d \lambda^j A_j = \prod_{j=1}^d(I + \lambda B_j)$,
we have $\sum_{j=1}^d \tilde{B}_j = n^{-1/2}A_1$ and $\sum_{1 \leq j < k \leq d}\tilde{B}_j\tilde{B}_k = n^{-1}A_2$.
Then, the right side of the last inequality can be rewritten and estimated
by Lemma~\ref{sums}  as follows:
\begin{align*}
\log n^{\frac d2} +& \frac{1}{n}\tr\real(\lambda A_1) + \frac{1}{2}|\lambda|^2\sum_{j=1}^d\norm{\tilde{B}_j}_F^2 \\
\leq&\log n^{\frac d2} + \frac{1}{n}\tr\real(\lambda A_1) + \frac{1}{2n}(\norm{A_1}_F^2 - 2\tr\real A_2)|\lambda|^2.
\end{align*}
Therefore, we obtain
\[
\log\norm{P(\lambda)}_F \leq \log n^{\frac d2} + \frac{1}{n}\tr\real(\lambda A_1) + \frac{1}{2n}(\norm{A_1}_F^2 - 2\tr\real A_2)|\lambda|^2.
\]
Taking exponent of both sides of the above inequality ends the proof.
\end{proof}

Below we present an example of a sequence of matrix polynomials $P_k(\lambda)$ of the same size $n\times n$
but increasing degree, based on the scalar polynomials in  \cite[Proof of Theorem 1.3]{Kne19}. 
The aim of this example is the following.
Note that the outcome inequality in Theorem~\ref{thm-fact1} has a factor  $n^{\frac d2}$ in front of the exponent.
We will show that this  factor needs to depend on high powers of $n$.

  \begin{example}\label{cmv}\rm 
Take two matrices
of size $n \times n$:
$A_1 = [1]_{s,t=1}^n$ and $A_2:=-I_n$.
Keeping the analogy with ~\cite{Kne19} we have:
\[
\Gamma := (A_1^2 - 2A_2)/2 = I + \frac{n}{2}A_1
\]
and
\begin{eqnarray*}
D_k := \Gamma - A_1^2/(2k) = I + \frac{n(k-1)}{2k}A_1\text{.}
\end{eqnarray*}
Now, let us define
$$
P_k(\lambda) :=
\biggl(I + \frac{A_1\lambda}{k}\biggr)^k\biggl(
I + \frac{\sqrt{D_k}\lambda}{\sqrt{k}}\biggr)^k\biggl(I - \frac{\sqrt{D_k}\lambda}{\sqrt{k}}
\biggr)^k
$$
and observe it satisfies the assumptions of Theorem~\ref{thm-fact1} ($d=3k$,
$A_1$ and $A_2$ equal the respective coefficients of $P_k$).
Using elementary calculations (see \cite[Example~4.7]{szymanski2024phd} for technical details) we obtain

\begin{equation}\label{mlim}
\lim_{k \to \infty}\norm{P_k(\ii y)}_F
= \Bigl[e^{2y^2}\Big(n-1+\big|e^{\ii yn}e^{y^2n^2/2}\big|^2\Big)\Bigr]^{1/2}= 
e^{y^2}(e^{n^2y^2} + n - 1)^{1/2}.
\end{equation}
Suppose that 
\[
\norm{P(\lambda)}_F \leq
c(n)\exp\Bigl(\frac{1}{n}\tr\real(\lambda A_1)
+ \frac{1}{2n}\big(\norm{A_1}_F^2 - 2\tr\real A_2\big)|\lambda|^2\Bigr), 
\quad \lambda \in \Comp,
\]
with some $c(n)$ depending on $n$ only. Setting $\lambda=\ii y$ we receive
\begin{equation}\label{mlim-e}
\norm{P_k(\ii y)}_F \leq
c(n)\exp\Big(
\frac1n\real(n \ii y) + \frac1{2n}(n^2+2n)|\ii y|^2 )\Big)= e^{(\frac n2+1)y^2},\quad y>0. 
\end{equation}

Then from \eqref{mlim-e} and \eqref{mlim} we receive
$$
c(n)\geq \frac{ e^{y^2}(e^{n^2y^2} + n - 1)^{1/2}}{e^{(\frac n2+1)y^2} }\geq \frac{ e^{y^2+ ny} }{e^{(\frac n2+1)y^2} } ,\quad y>0.
$$
Setting $y=1$ we obtain
\[ c(n)\geq e^{n/2}. \]

\end{example}

\section{Matrix polynomials and determinantal representation}\label{sec:factor} 

Recall that the complete factorisation of rational functions into elementary
factors (i.e., factors of the form ($I-\frac{1}{\lambda-\alpha}u_jv_j^*$)  
is fully characterised in terms of system realisation, see \cite[Theorem 10.5]{bart2007factorization}. Applying that to the  rational function $W(\lambda):=P(\lambda^{-1})$, where  $P(\lambda)=I+\sum_{k=1}^N\lambda^k A_k$, we obtain the   the following. 
\begin{proposition}\label{rev-elementary}
    Let $P(\lambda)=I+\sum_{j=1}^N\lambda^j A_j\in\CC{n}[\lambda]$ be matrix polynomial
with $A_N\neq 0$ and $\det(P(\lambda))$ not identically equal to $0$.
Fix $d\in\BN$.
Then the following conditions are equivalent:
\begin{enumerate}[\rm (i)]
\item $P(\lambda)=\prod_{j=1}^d(I+\lambda B_j)$ where each $B_j\in\CC{n}$ $(j=1,\dots,d)$ is of rank one;
\item $P(\lambda)$ has a realisation $P(\lambda)=I+\lambda C(I-\lambda A)^{-1} B$ such that $A$ is strictly upper-triangular, $A-BC$ is lower-triangular,
$A\in\CC{d}$, $B\in\BC^{d\times n}$, and $C\in\BC^{n\times d}$.
\end{enumerate}
Furthermore, if {\rm (ii)} holds, then denoting 
\begin{equation}\label{BC}
B=\begin{bmatrix} v_1^*\\ \vdots \\ v_d^* \end{bmatrix},\quad C= \begin{bmatrix} u_1 & \cdots & u_d\end{bmatrix}
\end{equation}
one has 
\begin{equation}\label{A}
A=\begin{bmatrix}  
0 & v_2^*u_1 & v_3^* u_1 & \cdots & v_d^*u_1 \\  
& 0 & v_3^* u_2 & \cdots & v_d^*u_2 \\
\vdots&&\ddots&\ddots&\vdots  \\
&&&0& v_d^* u_{d-1}\\
0&  &\cdots& &0\\\end{bmatrix}
\end{equation}
the matrices $B_j$ in (i) are given by $B_j=u_jv_j^*$, $j=1,\dots,d$, and $d\geq N$.
\end{proposition}
\begin{proof}
Proof of the implication $\rm(i)\Rightarrow(ii)$ follows from
\cite[Theorem 10.5]{bart2007factorization} applied to the rational function
$W(\lambda)=P(\lambda^{-1})$.
For the converse, we have to justify that $W(\lambda)$ does not
admit an elementary factor $(I+\frac 1{\lambda-\alpha}B_j)$ with
$\alpha\neq 0$. The proof of \cite[Theorem 10.5]{bart2007factorization}
shows that each $\alpha$ is a diagonal entry of $A$, hence
all the elementary factors of $W(\lambda)$ are indeed of the form $(I+\frac{1}{\lambda}B_j)$.
\end{proof}




Below $X\odot Y$ stand for the Hadamard product of two matrices.
\begin{theorem}\label{ABCISzasz}
Let $P(\lambda)\in\CC{n}[\lambda]$ be a matrix polynomial having a realisation as in  {\rm(ii)}
of Proposition~\ref{rev-elementary}. If 
\begin{equation}\label{elementary-pos}
\norm{CB}_F^2 +\real\tr\big((BC)\odot(BC)\big)\geq \tr\big((C^*C)\odot(BB^*)\big)+\real\tr\big((BC)^2\big)
\end{equation}
then the Sz\'asz inequality
\begin{equation}\label{SzABC}
\norm{P(\lambda)}_F \leq
n^{\frac d2}\exp\Bigl(\frac{1}{n}\tr\real(\lambda CB)
+ \frac{1}{2n}\big(\norm{CB}_F^2 + 2\tr\real (CAB)\big)|\lambda|^2\Bigr), 
\quad \lambda \in \Comp,
\end{equation}
holds. 
\end{theorem}
\begin{proof}
Knowing that $B^*=\begin{bmatrix}v_1&\ldots&v_d\end{bmatrix}$ and
$C=\begin{bmatrix}u_1&\ldots&u_d\end{bmatrix}$, by a direct computation
one obtains the following identities:
\begin{align*}
\tr( BC)^2 &=\sum_{1\leq j,k\leq d} (v_j^*u_k)(v_k^*u_j),\\
\norm{CB}_F^2=\tr(C^*CBB^*) &=\sum_{1\leq j,k\leq d} (u_k^*u_j)(v_j^*v_k).
\end{align*}
In order to exclude entries with $j=k$ we subtract traces of respective
Hadamard products. Then  Proposition~\ref{elem}~(a) shows that \eqref{elementary-pos} is equivalent to  \eqref{M-semis}.

Thus \eqref{M-semis} holds and application of Theorem~\ref{thm-fact1} ends the proof for \eqref{SzABC} is exactly \eqref{Sz2} written in terms of the
realisation.
\end{proof}

We present now two simple instances in which the  condition \eqref{elementary-pos} is satisfied.

\begin{corollary}
Let $P(\lambda)$ has a realisation $P(\lambda)=I+\lambda C(I-\lambda A)^{-1} B$  with $A,B,C$  given by \eqref{A} and \eqref{BC} respectively.
If, additionally, $C=\ii B^*+\Delta$ with $\Delta\in \Comp^{n\times d}$ of sufficiently small norm, then the Sz\'asz inequality \eqref{SzABC} holds.
\end{corollary}

\begin{proof}
In this setting, the statement (ii) of Proposition~\ref{rev-elementary} can be readily verified.
For $C=\ii B^*$ note that $B_j= \ii v_j v_j^*$, $j=1,\dots,d$. Hence \eqref{M-semis} is satisfied with a strong inequality.
This is equivalent to saying that \eqref{elementary-pos} holds with a strong inequality.
Adding a sufficiently small perturbation $\Delta$ preserves the inequality \eqref{elementary-pos} and in consequence proves \eqref{SzABC}. 
\end{proof}

\begin{corollary}
Let $P(\lambda)$ has a realisation $P(\lambda)=I+\lambda C(I-\lambda A)^{-1} B$
with $A,B,C$  given by \eqref{A} and \eqref{BC} respectively.
If, additionally, the entries of matrices $B$ and $\ii C$ are nonnegative
real numbers, then the Sz\'asz inequality \eqref{SzABC} holds.
\end{corollary}

\begin{proof}
Recall that $\tr(XY)\geq \tr (X\odot Y)$ for matrices with nonnegative entries. 
Taking $X=C^*C$, $Y=B^*B$ we receive
$$
\tr( C^*C BB^*) \geq \tr\big((C^*C)\odot(BB^*)\big).
$$
Taking $X=Y=\ii BC$ we receive
$$
\real\tr\big((BC)\odot(BC)\big)\geq \real\tr\big((BC)^2\big),
$$    
hence, \eqref{elementary-pos} is satisfied and in consequence \eqref{SzABC} holds.
\end{proof}

\section{Functional calculus in one variable}
In this section we prove inequalities \eqref{tri}, \eqref{sqrti} and \eqref{ui} providing bounds
for the norm of the scalar polynomial evaluated on the matrix variable. 
In Proposition~\ref{pro-funct1} the estimation depends on both the size $n$ of a matrix and the degree $d$
of a polynomial $p\in\Comp[\lambda]$. The next result (Proposition~\ref{pro-funct2}) provides
the estimation dependent on the parameter $d$ only, while the last Theorem~\ref{thm-funct3}
contains the estimation independent on both parameters $n$ and $d$.
At the end of the section we provide examples comparing the estimates and the von Neumann inequality.

Recall  the introduced in the previous section quantities $\Lambda=( b_1,\dots,  b_d)$,
$\var\Lambda$, $\mean\Lambda$, where $-b_1^{-1},\dots,-b_d^{-1}$ are the roots of
$p(\lambda)=1 + \sum_{j=1}^d a_j\lambda^j $. Our aim is to highlight here the interplay
between the roots of $p(\lambda)$  and the eigenvalues of $A$. 
Writing $p(\lambda) = \prod_{j=1}^d (1 +  b_j \lambda)$, we have the following customary identities,
which will be of frequent use:
\begin{equation}\label{alphaeq}
    \sum_{j=1}^d  b_j  = a_1,\quad \sum_{1 \leq j<k \leq d}  b_j  b_k  = a_2,
\end{equation}
and, under the assumption of \eqref{var-semis},
\begin{equation}\label{alphaeq2}|a_1|^2 - 2\real a_2 \geq |a_1|^2-2\real \sum_{1\leq j<k\leq d} b_j \overline{ b_k }
= \sum_{j = 1}^d | b_j |^2 = d\mean( |\Lambda|^2)  > 0.
\end{equation}

\begin{theorem}\label{pro-funct1}
Let $p(\lambda) = 1 + \sum_{j=1}^d a_j\lambda^j =\prod_{j=1}^d (1 +  b_j \lambda)$
be a  scalar polynomial with $b_j$'s satisfying the location condition \eqref{var-semis}.
Then the Frobenius norm of  $p(A)$ for any $A\in\CC{n}$ can be estimated as
\begin{equation}\label{e1}
\norm{p(A)}_F \leq
n^{\frac d2}\exp\Bigl(\frac{1}{n}\tr\real(a_1 A)
+ \frac{1}{2n}(|a_1|^2 - 2\real a_2)\norm A _F^2\Bigr).\end{equation}

\end{theorem}
\begin{proof}
Assume that $A\neq -\frac1{ b_j}  I$, $(j=1,\dots, d)$ and  use  $d$-times Lemma~\ref{mlog}. We have 
\begin{align*}
\log\norm{p(A)}_F 
&\leq \sum_{j=1}^d \log\norm{I +  b_j  A}_F \\
&= d\log\sqrt{n} + \sum_{j=1}^d \log\norm{\frac{1}{\sqrt{n}}I + \tilde{b}_j A}_F \\
&\leq \log n^{\frac d2} + \frac{1}{\sqrt{n}}\sum_{j=1}^d \tr\real(\tilde{b}_j A)
    + \frac{1}{2}\sum_{j=1}^d \norm{\tilde{b}_j A}_F^2\\
    &= \log n^{\frac d2}+ \frac{1}{n}\tr\real(a_1 A) + \frac{1}{2}\norm A _F^2\sum_{j=1}^d |\tilde{b}_j|^2
\end{align*}
where $\tilde{b}_j := n^{-1/2} b_j $ for $j \in \{1, 2, \dots, d\}$ and the identities \eqref{alphaeq} 
were used in the last step. 
We use now the inequality  \eqref{scalarv}, which follows from the assumption \eqref{var-semis}, and  estimate  further as follows
\[
\log\norm{p(A)}_F \leq \log n^{\frac d2} + \frac{1}{n}\tr\real(a_1 A)
    + \frac{1}{2n}(|a_1|^2 - 2\real a_2)\norm A _F^2.
\]
Taking exponent on the both sides  ends the proof of \eqref{e1} for $A\neq -1/ b_j  I$,
$j=1,\dots, d$. By continuity, the inequality \eqref{e1} holds for any $A$.

\end{proof}

As a corollary, we obtain the following absolute bound. 

\begin{corollary}
Let $p(\lambda) = 1 + \sum_{j=1}^d a_j\lambda^j =\prod_{j=1}^d (1 +  b_j \lambda)$ be
a  scalar polynomial with $b_j$'s satisfying the location condition \eqref{var-semis}
and let $A\in \CC{n}$ be a matrix satisfying  

 \begin{equation}\label{eq:exp0}
 |a_1|^2  \norm{A}_F^2 + 2\real\Big(a_1 \tr A- a_2\norm{A}_F^2 \Big) \leq 0.
 \end{equation}
Then $\norm{p(A)}_F\leq n^{\frac d2}$.
\end{corollary}

In the next results we move on to estimating the operator norm of $p(A)$ in two different ways.

\begin{theorem}\label{pro-funct2}
Let $p(\lambda) = 1 + \sum_{j=1}^d a_j\lambda^j =\prod_{j=1}^d (1 +  b_j \lambda)$ be a  scalar polynomial such that $b_j$'s satisfy the location condition \eqref{var-semis}.
Then the induced two-norm of  $p(A)$ for any $A\in\CC{n}$ can be estimated as follows:
\begin{eqnarray}
\label{pro-funct2a}  \norm{p(A)}
  &\leq& \exp\Bigl( \norm A \sum_{j=1}^d |b_j| \Bigr)\\
\label{pro-funct2b}
  &\leq& \exp\Bigl(\norm A \sqrt{d(|a_1|^2 - 2\real a_2)}\Bigr),
\end{eqnarray}
and
\begin{eqnarray}\label{thm-funct3}
\norm{p(A)} &\leq& \exp\bigl(|a_1|\norm A  + {\textstyle\frac12} (|a_1|^2 -2 \real(a_2))\norm A^2\bigr).    
\end{eqnarray}
\end{theorem}
\begin{proof}

Applying \eqref{imm}
 $d$ times we obtain 
\begin{equation}
\log\norm{p(A)}
\leq
\sum_{j = 1}^d\log \norm{I +  b_j  A} 
\leq \sum_{j = 1}^d \norm{ b_j  A}
=\norm A \sum_{j = 1}^d | b_j | ,
\end{equation}
which shows \eqref{pro-funct2a}. The inequality of means and subsequently \eqref{scalarv} together with \eqref{alphaeq2} allow us to estimate the latter by
\begin{equation}\nonumber
\norm A \sum_{j = 1}^d | b_j |
\leq \norm A\Bigl(d\sum_{j = 1}^d | b_j |^2\Bigr)^{1/2} 
\leq \norm A\sqrt{d(|a_1|^2 - 2\real a_2)},
\end{equation}
which gives \eqref{pro-funct2b}. 

The proof of \eqref{thm-funct3} follows from the following version of the von Neumann inequality
\begin{equation}\label{vN}
 \norm{p(A)}\leq \sup_{|z|\leq \norm A} |p(z)|, 
\end{equation}
where $p$ is a polynomial and $A$ is an arbitrary  square matrix.
It can be easily derived from the usual von Neumann inequality by considering $\tilde p(z)=p(z\cdot\norm A)$ and  $\tilde A=A/\norm A$.
Using the scalar
Sz\'asz inequality \eqref{dB} we obtain
\begin{align*}
\norm{p(A)}&\leq \sup_{|z|\leq \norm A} |p(z)| \\
&\leq  \sup_{|z|\leq \norm A} \exp\bigl(\real (a_1z)  + \frac12 (|a_1|^2 -2 \real(a_2))|z|^2\bigr) \\
&= \exp\bigl(|a_1|\norm A  + \frac12 (|a_1|^2 -2 \real(a_2))\norm A^2\bigr).
\end{align*}

\end{proof}

Let us recall that for any matrix $A \in \CC{n}$ the relation
$\norm{A}\leq \norm{A}_F \leq \sqrt{n}\norm{A}$
between the operator norm and the Frobenius norm holds. 
Below, we make a comparison of the global bounds presented in this section.
The inequalities resulting from Theorems~\ref{pro-funct1} and ~\ref{pro-funct2} are listed below:
\begin{align}
\label{tri} \norm{p(A)}_F &\leq
    n^{\frac d2}\exp\Bigl(\frac{1}{n}\tr\real(a_1 A) + \frac{1}{2n}(|a_1|^2 - 2\real a_2)\norm{A}_F^2\Bigr)
    =:E_1, \\
\label{sqrti} \norm{p(A)} &\leq
    \exp\Bigl(\norm{A}\sqrt{d(|a_1|^2 - 2\real a_2)}\Bigr)
    =:E_2, \\
\label{ui} \norm{p(A)} &\leq
    \exp\bigl(|a_1|\norm{A} + \frac{1}{2}(|a_1|^2 -2 \real a_2)\norm{A}^2\bigr)
    =:E_3.
\end{align}
Below we provide three examples where each of these inequalities gives the best bound. 
 The drawback of \eqref{tri} lies in the dependence on $n$ and $d$. However, it appears that \eqref{tri}
can beat the von Neumann inequality in some cases, while the other two inequalities can not (see Example~\ref{ex:vN}). 
The key issue is here the term $\tr\real(a_1 A)$, which might be negative, if the zeros of $p(\lambda)$ and eigenvalues of $A$ are conveniently located.

\begin{example}\label{comparison}
All the examples are with $n=2$ and $d=3$.

\begin{enumerate}[\ a)]
\item $E_1 < E_2 < E_3$ implies that \eqref{tri}
provides the lowest estimation for $\norm{p(A)}$.
It is also evident that it gives then the best estimation
for  $\norm{p(A)}_F$.

To witness this scenario, take 
\begin{equation}\label{ex-pA-1}
p(\lambda) = -(\lambda - 1)^3
= -\lambda^3 + 3\lambda^2 - 3\lambda + 1,\qquad 
A = \begin{bmatrix} 2 & 0\\ 0 & 0
\end{bmatrix}.
\end{equation}
The polynomial $p$ is clearly stable with respect to the
upper half-plane and hence satisfies \eqref{var-semis}.
The right-hand-side values in \eqref{tri} -- \eqref{ui}
are now $E_1=2\sqrt{2}$, $E_2=e^6$ and $E_3=e^{12}$.

\medskip\item $\sqrt2E_2 < E_1 < E_3$ guarantees that
\eqref{sqrti} gives the lowest upper bound for $\norm{p(A)}$ and $\norm{p(A)}_F$ as well.

To obtain the above one can take
\[
p(\lambda) = -(\lambda - 1)^3
\text{ and }
A = \begin{bmatrix} -2 & 0\\ 0 & 0 \end{bmatrix}.
\]
The estimations from \eqref{tri} -- \eqref{ui} attain
$E_1=2e^6\sqrt{2}$, $E_2=e^6$ and $E_3=e^{12}$.

\medskip\item $\sqrt2 E_3 < \sqrt{2}E_2 < E_1$ is the case when
\eqref{ui} provides the best upper bound for $\norm{p(A)}_F$.
In consequence also the bound for $\norm{p(A)}$ is the sharpest.

The last ordering is attained when considering
\[
p(\lambda) = (\lambda - 1)^2(\lambda + 1) = \lambda^3 - \lambda^2 - \lambda + 1
\text{ and }
A = - I_2.
\]
The above polynomial is clearly stable with respect to the
upper half-plane.
From \eqref{tri} -- \eqref{ui} we obtain values
$E_1=2e^{5/2}\sqrt{2}$, $E_2=e^3$ and $E_3=e^{5/2}$.
\end{enumerate}

\end{example}

\begin{example}\label{ex:vN}
It is worth noting that a direct use of von Neumann inequality \eqref{vN} to the matrix and polynomial form \eqref{ex-pA-1}
gives a worse estimation, since
\[ \sup_{|\lambda|\leq \norm{A_1}} |p_1(\lambda)|=
\sup_{|\lambda|\leq 2} |p_1(\lambda)|=3^3  < 2\sqrt 2 = E_1. \]
Therefore \eqref{tri} provides a sharper bound for both
$\norm{p_1(A_1)}_F$ and $\norm{p_1(A_1)}$.

The inequality \eqref{ui} is proven using \eqref{vN}, hence cannot
provide a better estimation. For \eqref{sqrti} we have
\[
E_2 \geq \sup_{|\lambda|= \norm A} \norm{p(\lambda I)}
= \sup_{|\lambda|\leq \norm A} |p(\lambda)|.
\]
This argument actually shows that no function of $\norm A$ and $p$
can provide a lower estimation for $\norm{p(A)}$
that the von Neumann inequality.
\end{example}


\section{Other forms of functional calculus}
As one may observe, Theorem~\ref{thm-funct3} 
results from the linking of two inequalities: von Neuman and Sz\'asz.
Using this technique, one may obtain similar results in the multivariate case.
We list here some of them, the proofs follow the same steps as the proof of
Theorem~\ref{thm-funct3} above and are left to the reader.
We also present a few corollaries on univariate matrix
polynomials of a matrix variable.

\begin{remark}    
A Sz\'asz-type inequality for multivariate scalar polynomials can be found in the work of Borcea and Br\"and\'en
\cite{BorB09} and was improved by Knese \cite[Theorem 1.4]{Kne19} later on. %
Employing Ando's \cite{ando63} inequality we obtain the following estimation
in terms of $p_j:=\frac{\partial p}{\partial z_j}(0,0)$ and
$p_{j,k}:=\frac{\partial^2 p}{\partial z_j\partial z_k}(0,0)$
($1\leq j,k\leq d$). For any $p\in\Comp[z_1,z_2]$,
stable with respect to the open upper half plane with  $p(0,0)=1$ and
$A_1,A_2\in\CC{n}$ with $A_1A_2=A_2A_1$ we have 
\[
\norm{p(A_1,A_2)}\leq \exp \bigg(\!
\norm{A_1}|p_1|+\norm{A_2}|p_2|
+\frac 12\norm{\bf A}^2_\infty\Big(
\big|p_1+p_2 \big|^2
-\real\!\sum_{j,k=1}^2p_{j,k}
\Big)
\bigg),
\]
where $\norm{\bf A}_\infty:=\max\{\norm{A_1},\norm{A_2}\}$.
\end{remark}

\begin{remark}\label{Bre}
Using a subsequent result by Knese  \cite[Theorem 1.5]{Kne19} and Brehmer's \cite{Breh61} multivariate version of von Neumann's inequality (see also \cite{Lub78})
we receive the following. For  $p\in\Comp[z_1,\ldots,z_k]$  stable with respect to $H_0^k$
such that $p(\mathbf{0})=1$ and  $\mathbf{A}=(A_1,\ldots,A_k)\in(\CC{n})^k$  a commuting $k$-tuple we have
\[
\norm{p(\mathbf{A})}\leq \exp \Big[ \sqrt{2k} \norm{\nabla p(\mathbf 0)}\norm{\mathbf A}_2+
k\big(\norm{\nabla p(\mathbf 0)}^2 + \norm{ \real H_p(\mathbf 0) } \big) \norm{\mathbf A}_2^2
\Big],
\]
where $\norm{\mathbf A}_2=\sqrt{\sum_{j=1}^k\norm{A_j}^2}$.
\end{remark}

\begin{remark}
Another possibility is to use the  constants $C_{k,n}\geq 1$ introduced by Hartz, Richter and Shalit \cite{HarR22}, who showed that   for a commuting row contraction
$\mathbf T= (T_j)_{j=1}^k\in(\CC{n})^k$ (i.e.\ $\|\sum_{j=1}^n T_jT_j^*\|\leq 1$) the following inequality holds
\begin{equation*}\label{vN-Ckn}
|p(T_1,\ldots,T_k)| \leq C_{k,n} \sup\{|p(\boldsymbol z)|:\norm{\boldsymbol z}\leq 1, j=1,\ldots,k \}.
\end{equation*}
In consequence, for  $p\in\Comp[z_1,\ldots,z_k]$  stable with respect to $H_0^k$ and satisfying $p(\mathbf 0)=1$
a commuting tuple of matrices $\mathbf T\in(\CC{n})^k$ we have
\begin{equation*}\label{n-var-Ckn}
\norm{p(\mathbf A)}\leq C_{k,n}
\exp \Big[ \sqrt{2} \norm{\nabla p(\mathbf 0)}\norm{\mathbf A}_R+
\big(\norm{\nabla p(\mathbf 0)}^2 + \norm{ \real H_p(\mathbf 0) } \big) \norm{\mathbf A}_R^2
\Big],
\end{equation*}
where 
$$
\norm{\bf A}_R:=\sqrt{\Big\|\sum_{j=1}^k A_jA_j^*\Big\|}\leq 1,
$$
note that $\norm{\bf A}_R\leq \norm{\bf A}_2$.
Moreover, for $n\leq 3$, due to the results of Kosi\'nski \cite{Kosin} and Knese \cite{Kne16}  we have:
\[\norm{p(\mathbf A)}\leq \exp \Big[ \sqrt{2} \norm{\nabla p(\mathbf 0)}\norm{\mathbf A}_2+
\big(\norm{\nabla p(\mathbf 0)}^2 + \norm{ \real H_p(\mathbf 0) } \big) \norm{\mathbf A}_2^2.
\Big]
\]
In this case the bound is essentially better  than  the one obtained using Lubin's 
result in Remark~\ref{Bre} above. 
\end{remark}

\begin{remark}\label{pro-complete}
Now we present a complete version of the Sz\'asz inequality.
It is known (see e.g.\ \cite[p.54]{NagyFoiasBK2010}, \cite{BadBec14})
that the unit disc is a complete spectral set.
Together 
with Theorem~\ref{thm-fact1} it leads to the following. 
Let $P(\lambda) = I_n + \sum_{j=1}^N \lambda^j A_j
\in \CC{n}[\lambda]$ be a matrix polynomial with the following factorisation
$P(\lambda) = \prod_{j=1}^d(I_n + \lambda B_j)$ satisfying \eqref{M-semis}.
Then, for $T\in\CC{m}$, the following inequality holds
\[ 
\biggl\| I_n\otimes I_m + \sum_{j=1}^{N} A_j\otimes T^j \biggr\|\leq
n^{\frac d2}\exp\Bigl(\frac{1}{n}|\tr A_1| \norm{T}
+ \frac{1}{2n}\big(\norm{A_1}_F^2 - 2\tr\real A_2\big)\norm{T}^2\Bigr).
\]
\end{remark}

\begin{remark}\label{pro-mlak}

Another  inequality was shown by Mlak \cite{Mlak71}
(cf.\ \cite{BadBec14}).
If $P(\lambda)=A_0+A_1\lambda+\ldots+A_d\lambda^N \in \CC{n}[\lambda]$ is a matrix polynomial 
and a contraction $T\in\CC{n}$ doubly commutes with coefficients of $P$,
i.e.\ $TA_j=A_jT$, $TA_j^*=A_j^*T$ for $j=0,\ldots,N$, then defining 
$$
P(T):=A_0+A_1T+A_2T^2+\ldots+A_dT^N
$$
we have
\[ \norm{P(T)} \leq \sup\{\norm{P(z)}:|z|\leq 1\}.\]

Let $P(\lambda) = I_n + \sum_{j=1}^N \lambda^j A_j
\in \CC{n}[\lambda]$ be a matrix polynomial with the following factorisation
$P(\lambda) = \prod_{j=1}^d(I_n + \lambda B_j)$ satisfying \eqref{M-semis}.
For $T\in\CC{n}$ doubly commuting with every $A_j$ the norm of
$P(T)$ can be estimated from above as
\[
\norm{P(T)} \leq
n^{\frac d2}\exp\Bigl(\frac{1}{n}|\tr A_1|\norm{T}
+ \frac{1}{2n}\big(\norm{A_1}_F^2 - 2\tr\real A_2\big)\norm{T}^2\Bigr).
\]
 Note that if $T$ doubly commutes with $B_j$ for $j=1,\ldots,d$ then it doubly commutes
with all the $A_j$'s as well. 
\end{remark}

\begin{remark}
In \cite{hartz2023ax}, Hartz provides a matrix polynomial version of
von Neumann inequality, namely for $P(\lambda)=A_0+A_1\lambda +\ldots A_d\lambda^N\in\CC{n}[\lambda]$
and a contraction $T\in\CC{n}$ the following inequality holds
\begin{equation*}
\norm{P(T)}\leq \sqrt{N+1}\sup\{\norm{P(\lambda)}\colon |\lambda|\leq 1\}.
\end{equation*}
In consequence, we obtain the following.
Let $P(\lambda)=I_n+\sum_{j=1}^N A_j\lambda^j\in\CC{n}[\lambda]$ be a matrix polynomial
admitting factorisation $P(\lambda)=\prod_{j=1}^d(I_n+\lambda B_j)$ and satisfying the
condition \eqref{M-semis}. Then for $T\in\CC{n}$ (no assumption on commutation required)
we have the following estimation:
\[
\norm{P(T)} \leq
\sqrt{n^d(N+1)}\exp\Bigl(\frac{1}{n}|\tr A_1|\norm{T}
+ \frac{1}{2n}\big(\norm{A_1}_F^2 - 2\tr\real A_2\big)\norm{T}^2\Bigr).
\]
\end{remark}

\bibliographystyle{plain}
\bibliography{szw}

\end{document}